\documentclass[review, 11pt]{elsarticle}

\usepackage{lineno,hyperref}
\modulolinenumbers[5]

\usepackage[utf8]{inputenc}
\usepackage[T1]{fontenc}
\usepackage{enumerate}
\usepackage{hyperref}
\hypersetup{
  colorlinks   = true,
  citecolor    = blue,
  linkcolor    = blue
}

\usepackage{amsmath,amsfonts,amsthm,amssymb,color,tikz, comment}
\usepackage{mathrsfs}   
\usepackage[normalem]{ulem}
\usepackage{pdfsync}
\usepackage[font={scriptsize}]{caption}
\usepackage{comment}

\usepackage{graphicx}
\usepackage{subcaption}

\usepackage[left=1in, right=1in, top=1.1in,bottom=1.1in]{geometry}
\DeclareMathAlphabet\mathbfcal{OMS}{cmsy}{b}{n}

\newcommand{\der}{\delta}

\newcommand{\E}{\mathbb E}
\newcommand{\R}{\mathbb R}

\newcommand{\be}{\mathbf{E}}

\newcommand{\bp}{\mathbf{P}}

\newcommand{\ce}{\mathcal E}
\newcommand{\cf}{\mathcal F}

\newcommand{\cp}{\mathcal P}

\newcommand{\ep}{\varepsilon}

\newcommand{\ga}{\gamma}

\newcommand{\la}{\lambda}
\newcommand{\laa}{\Lambda}
\newcommand{\om}{\omega}
\newcommand{\oom}{\Omega}

\newcommand{\si}{\sigma}

\newcommand{\tte}{\Theta}

\newcommand{\lp}{\left(}
\newcommand{\rp}{\right)}
\newcommand{\lc}{\left[}
\newcommand{\rc}{\right]}
\newcommand{\lcl}{\left\{}
\newcommand{\rcl}{\right\}}
\newcommand{\lln}{\left|}
\newcommand{\rrn}{\right|}

\newtheorem{theorem}{Theorem}[section]

\newtheorem{assumption}[theorem]{Assumption}

\newtheorem{corollary}[theorem]{Corollary}

\theoremstyle{remark}
\newtheorem{remark}[theorem]{Remark}

\theoremstyle{remark}

\newcommand{\bean}{\begin{eqnarray*}}
\newcommand{\eean}{\end{eqnarray*}}
\newcommand{\ben}{\begin{enumerate}}
\newcommand{\een}{\end{enumerate}}
\newcommand{\beq}{\begin{equation}}
\newcommand{\eeq}{\end{equation}}

\begin{document}

\begin{frontmatter}

\title{A Many-Server Functional Strong Law For A Non-Stationary Loss Model}

%
%
%
\author[paddress]{Prakash Chakraborty}
\ead{chakra15@purdue.edu}
\author[haddress]{Harsha Honnappa}
\ead{honnappa@purdue.edu}

\address[paddress]{Department of Statistics, Purdue University, West Lafayette IN, U.S.A.}
\address[haddress]{School of Industrial Engineering, Purdue University, West Lafayette IN, U.S.A.}
\begin{abstract}
The purpose of this note is to show that it is possible to establish a many-server functional strong law of large numbers (FSLLN) for the fraction of occupied servers (i.e., the scaled number-in-system) without explicitly tracking either the age or the residual service times of the jobs in a non-Markovian, non-stationary loss model. This considerable analytical simplification is achieved by exploiting a semimartingale representation. The fluid limit is shown to be the unique solution of a Volterra integral equation.
\end{abstract}


\end{frontmatter}


\section{Introduction}
This note establishes a functional strong law of large numbers (FSLLN) for a non-Markovian, non-stationary $M_t/G/n/n$ loss model in the many-server limit as $n \to \infty$. Stationary loss models have been studied extensively, with the Erlang-B formula being a cornerstone consequence of this literature. Non-stationary models, on the other hand, are much harder to analyze, and closed form expressions are almost impossible to derive. Stochastic process approximations are therefore crucial for performance analysis of loss models. There is a significant body of work focused on establishing fluid approximations in many-server settings, though the methods are non-trivial. In the formative paper~\cite{kaspi}, the elapsed waiting time (or `age') of the jobs in the system are tracked, using which it is possible to obtain a martingale representation of the number-in-system process that yields the desired FSLLN for a $G/GI/n/\infty$ queue. In contrast,~\cite{zhang} develops a method where the residual service times of jobs in the $G/GI/n/\infty$ queue are tracked, in which case it is possible to establish the fluid limit without recourse to a martingale representation. Of course, while the latter approach in essence assumes that the service times are known at the time of arrival, as commented on in~\cite{zhang} the approach offers significant analytical simplification.

 The purpose of this note is to show that it is possible to establish a many-server FSLLN for the fraction of occupied servers (i.e., the scaled number-in-system) in a $M_t/G/n/n$ loss model, without explicitly tracking either the age or the residual service times of the jobs. We capitalize on the fact that this process is the sum of a pure jump bounded semimartingale and a bounded finite variation process. Indeed, we show that in the many server limit the fraction of occupied servers converges to the solution of a non-linear Volterra integral equation by exploiting the fact that the semimartingale is uniformly zero and the bounded finite variation process converges to a deterministic function. This semimartingale representation is natural and allows us to avoid tracking the residual service times. Our proofs are also considerably simpler and easier to follow. Consequently, we anticipate that our analysis will be intuitive and useful for a broad range of applications. For instance, as a consequence of our main result, we present a fluid limit for the fraction of arrivals that are blocked. This result can be used as a proxy for the blocking probability in the many-server limit. A crucial motivation for this paper is the need to develop `transitory fluid' traffic models; i.e., systems where a finite volume of jobs (in a continuum) enter a system over time. Queueing models fed by this type of traffic have been studied in~\cite{chakra,honnappa,bet} -- however all of these consider discrete-event models of traffic. Transitory {\it fluid} traffic models have not been studied in the literature, and would be of considerable use in the modeling of capacitated energy storage systems and high-speed computer networks. As an auxiliary result, therefore, we also establish a FSLLN for the integrated fraction of occupied servers. This process is a non-decreasing stochastic fluid with a maximum rate of increase. This type of model can be used to model the energy production from a solar array, for instance. 
 
 
 While our proof of the main result is not complicated, some commentary is in order. We consider a sequence of $M_t/G/n/n$ models with nonstationary Poisson traffic with deterministic intensity $(\lambda^n(t) : t \geq 0)$ where $\lambda^n(\cdot) = n \lambda(\cdot)$ and stationary general service times with finite first moment. We assume that the traffic and service processes are statistically independent of each other for every $n \geq 1$. Thus, the number of servers (i.e., the system capacity) is in scale with the arrival intensity. Using the fact that the fraction of occupied servers in the $n$th system can be represented as a stochastic integral with respect to a random counting measure, we extract the desired semimartingale representation. In Theorem~\ref{thm:fluid-lim} we show that the fraction of occupied servers converges to a deterministic limit. Identifying the limit function itself turns out to be a little tricky, {owing to the fact that the fraction of occupied servers is the solution of a discontinuous stochastic integral equation}. In order to identify the limit, we smooth the representation of the process by using a mollifier of the discontinuity. This allows us to identify the limit function as the solution of a specific non-linear Volterra integral equation. Finally, to establish uniqueness of the limit, we exploit the fact that the first time that the limit function hits the level $1$ (i.e., the system fluid level is full) is unique, from which it follows recursively that the first (and subsequent) times the limit leaves the fully occupied level and/or (re)enters state $1$ are unique. 

 \paragraph{\bf Related Literature}
There is a significant body of work establishing many-server fluid limits for stationary and non-stationary models, both with and without abandonment, starting with the seminal work in~\cite{halfin}; see~\cite{whitt1} for a recent survey. Our work is related to the development of proof techniques for many-server limits, and to work on approximations to nonstationary loss models. As noted before, Kaspi and Ramanan established a fluid limit for the number-in-system process in the formative paper~\cite{kaspi}, using a martingale representation of extracted using the elapsed waiting time or age of the jobs in the system. \cite{reed} on the other hand established a fluid (and diffusion) limit for the number-in-system process of a stationary $G/GI/n$ queue by using a representation of the number in system process that is similar to the system equations of a $G/GI/\infty$ queue. By establishing a link between the equations,~\cite{reed} was able to prove both a FSLLN and a functional central limit theorem (FCLT). Our approach is similar, in the sense that we exploit a random measure representation of the number-in-system process akin to system state representations in infinite server queues. Note that since we focus on nonstationary loss models, our representation is different from that of~\cite{reed}. While the analysis of models without abandonment are most relevant to our setting,~\cite{zhang} analyzed the number in system process of a $G/GI/n+GI$ queue by tracking the residual service times. In the nonstationary setting, in a series of papers~\cite{lu1,lu2} Lu and Whitt proved a fluid limit for a $G_t/GI/n+GI$ queue that experiences alternating periods of overload and underload, by tracking the age of the jobs in the system {\it a la}~\cite{kaspi}. More broadly, there has been a growing body of work on nonstationary loss models and various approximations, particularly for computing blocking probabilities~\cite{pender1,pender2,whitt1,whitt2,whitt3}. Our results complement these works by providing fluid limits that characterize the fraction of arrivals that encounter a blocked system.
 
\section{Preliminaries}
In this section we present some preliminary results that will be useful later on.
\subsection{Right continuous functions.}
Let $D=D[0,T]$ denote the space of right continuous functions on $[0,T]$ that have left limits. For a function $f \in D$ and $T_0 \subset [0,T]$, let 
$$
w_f(T_0) = \sup \lcl \lln f(t) - f(s) \rrn : s,t \in T_0 \rcl.
$$
Now, for $\der \in (0,T)$ let
$$
w_f'(\der) = \inf_{\cp: {\|\cp\| \leq \der}}~ \max_{0 < i \leq |\cp|} w_f([t_{i-1}, t_i)),
$$
where $\cp$ runs over the set of all partitions of $[0,T]$, in the sense that a generic $\cp$ looks like
$$
\cp = \lcl 0=t_0, \ldots, t_{|\cp|} = T \rcl,
$$
and $\|\cp\|$ denotes the mesh or norm of the partition $\cp$:
$$
\|\cp\| = \max_{1\leq i < |\cp|} \lln t_i - t_{i-1} \rrn.
$$
It can be shown that a function $f$ lies in $D$ if and only if
$$
\lim_{\der \downarrow 0} w_f'(\der) = 0.
$$
The proof of this result and related discussion can be found in \cite[Ch. 14]{billingsley}. The Skorohod distance between two functions $f$ and $g$ in $D$ is defined by
\begin{align*}
d_S (f,g) = \inf &\left\lbrace \ep > 0: \exists \text{ strictly increasing function }\la: [0,T] \mapsto [0,T], \text{ and } \right.\\
&\left.\sup_{t \in [0,T]}\lln \la(t) - t \rrn \leq \ep, ~\sup_{t \in [0,T]} \lln f(\la(t)) - g(t) \rrn \leq \ep \rcl.
\end{align*}
This topology created on $D$ by the Skorohod distance is the Skorohod topology.
\begin{theorem}
A set $A \subset D[0,T]$ has compact closure in the Skorohod topology if and only if:
$$
\sup_{f \in A} \sup_{t \in [0,T]} \lln f(t) \rrn < \infty,
$$
and
$$
\lim_{\der \downarrow 0} \sup_{f \in A} w_{f}'(\der) = 0.
$$
\end{theorem}
\begin{remark}
It can be shown that $D$ is not a complete space with respect to the Skorohod distance $d_S$ but there exists a topologically equivalent metric $d_0$ with respect to which $D$ is complete.
\end{remark}

\subsection{Counting measure.} 
Let $(\oom, \cf, \mathbfcal{F}= ( \cf_t )_{t \geq 0}, P)$ be a filtered probability space. Let $(N_t)_{t \geq 0}$ be a point process given by a sequence $(T_n)_{n \geq 0}$ of jump times, that is
$$
N_t = \sum_{i=1} \mathbf{1}_{\{T_i \leq t\}}.
$$
Suppose in addition the $n^{\textnormal{th}}$ jump time or arrival $T_n$ has a corresponding random variable $Z_n$ taking values in some measurable space $(E, \ce)$. Then $(T_n, Z_n)_{n \geq 1}$ is called an $E$-marked point process. For each $A \in \ce$, let the counting process $N_t(A)$ be given by
$$
N_t(A) = \sum_{i=1}^{\infty} \mathbf{1}_{\{Z_n \in A\}} \mathbf{1}_{\{T_n \leq t\}},
$$
and the corresponding counting measure $p(dt \times dz)$ by
$$
p(\om, (0,t] \times A) = N_t(\om, A).
$$
This means that for functions $H: \oom \times [0,\infty) \times \R \mapsto \R$
\beq\label{eq:count-mre-rep}
\int_0^t \int_{\R} H(\om, u, x) p(\om, du \times dx) = \sum_{i=1}^{\infty} H(\om, T_i(\om), Z_i(\om)) \mathbf{1}_{\{T_i(\om) \leq t\}}.
\eeq
For a point process $(N_t)_{t \geq 0}$, its intensity with respect to a given filtration $(\cf_t)_{t \geq 0}$ is given by
$$
\la_t = \lim_{\der \downarrow 0} P(N(t+\der t) - N(t) \vert \cf_t), \quad t > 0.
$$ 
If $(Z_n)_{n \geq 1}$ and $(T_n)_{n \geq 1}$ are independent, and $(Z_n)_{n \geq 1}$ are independent and identically distributed (iid) from a distribution with density $\nu$, then it is easy to see that the intensity of the marked point process $N_t(A)$ for some $A \in \ce$ is given by
\beq\label{eq:intensity-1}
\la_t(A) = \la_t \nu(A).
\eeq
We now say that $p(dt \times dz)$ admits the intensity kernel $\la_t \nu(dz)$. Let $\cp(\mathbfcal{F})$ denote the predictable $\si$-field on $\oom \times (0, \infty)$. Then for any mapping $H: \oom \times (0, \infty) \times E \mapsto \R$, measurable with respect to $\cp(\mathbfcal{F}) \otimes \ce$ satisfies the following projection result (cf. \cite[T3 Theorem, pp 235]{bremaud})
\beq\label{eq:proj}
E\lc \int_0^{\infty} \int_E H(s,z)p(ds \times dz) \rc = E \lc \int_0^{\infty} \int_E H(s, z) \la_s \nu(dz) ds  \rc.
\eeq
Thus defining the compensated measure 
$$
q(ds \times dz) = p(ds \times dz) - \la_s \nu(dz) ds,
$$
we have for every $H$ as in \eqref{eq:proj} that $\int_0^t \int_E H(s,z) q(ds \times dz)$ is a $(P, \cf_t)$ local martingale.

\subsection{Notations} We denote convergence in probability by $\stackrel{p}{\rightarrow}$, convergence uniformly on compact intervals and in probability by $\stackrel{ucp}{\rightarrow}$.

\section{Model and Results}
\subsection{Description of model.} We now introduce our model along with a useful representation of our main quantity of interest.
\begin{assumption}\label{assum:model}
Consider a $M_t/G/n/n$ loss model; namely, a queueing model with 
\begin{enumerate}
\item[i.] a non-homogeneous Poisson arrival process $A^n$ with rate $n \la$, where $\la$ is locally integrable;
\item[ii.] general service times sampled iid from a distribution $F$ with density $\nu$; and,
\item[iii.] $n$ servers and zero buffer.
\end{enumerate}
\end{assumption}

Let $\mathbf{R} = (R_i)_{i \geq 1}$ be the marked process where $R_i = (T_i, S_i)$, $T_i$ are the arrival time epochs corresponding to the arrival process $A^n$, and $S_i$ denotes the corresponding service time sampled iid from $F$. We assume that relevant random variables for every $n$ sit in a common filtered probability space $(\Omega, \cf, \mathbfcal{F}=(\cf_t)_{t>0}, \bp)$. Let $p^R(du, dx)$ denote the counting measure associated with the process $R$. Recall from \eqref{eq:count-mre-rep}, $p^R(du, dx)$ is a random measure on $[0, \infty) \times \R^{+}$ such that for functions $W:\Omega \times[0, \infty) \times \R \mapsto \R$ we have:
\beq\label{eq:R-mre}
\int_0^t \int_{\R} W(\om; u, x) p^R(du, dx) = \sum_{i=1}^{\infty} W(\om; T_i(\om), S_i(\om)) \mathbf{1}_{\{T_i(\om)\leq t\}}.
\eeq
Moreover note from \eqref{eq:intensity-1}, $p^R$ is a random measure with intensity 
\beq\label{eq:compensator}
p_c^R (du, dx) = n \la_u \nu(x) du dx.
\eeq
Denote $p_{\ast}^R$ to be the compensated random measure:
\beq\label{eq:p*}
p_{\ast}^R = p^R-p_c^R.
\eeq
\begin{remark}
We explain the need for Poisson arrival processes in our considerations. In the sequel we would need {finiteness of} the second moment {of stochastic integrals of bounded predictable processes with respect to the compensated measure, that is}
\beq\label{eq:2nd-moment}
\be {\lp \int_0^t \int_{\R} W(u,x)p_{R}^{\ast}(du,dx) \rp}^2< \infty
\eeq 
for a bounded predictable process $W$. This is well established when $p_{\ast}^R$ results from Poisson arrivals. However, we note that this is the only crucial requirement and all the results stated in this article hold true for any arrival process satisfying \eqref{eq:2nd-moment}.
\end{remark}
\subsection{Fraction of occupied servers}
Let $\rho_t^n$ denote the fraction of occupied servers at time $t$. Observe that the number of busy servers at time $t$ is the cumulative sum of arrivals at times $u$, $u \in [0,t]$ satisfying:
\begin{enumerate}
\item[(i)] the number of occupied servers at time $u$ is less than $n$.
\item[(ii)] the corresponding service requirement exceeds $t-u$.
\end{enumerate} 
Consequently we have:
\beq\label{eq:rho-1}
\rho_t^n = \dfrac{1}{n} \sum_{i=1}^{\infty} \mathbf{1}_{\{\rho_{T_i}^n < 1\}} \mathbf{1}_{\{S_i > t-u\}} \mathbf{1}_{\{T_i \leq t\}}
\eeq
Using \eqref{eq:R-mre}, the right hand side of \eqref{eq:rho-1} can be expressed as a stochastic integral with respect to the counting measure $p^R$:
\beq\label{eq:rho-2}
\rho_t^n = \int_0^t \int_{\R} W_n({t},u,x) p^R(du, dx),
\eeq
where
$$
W_n({t},u,x) = \dfrac{1}{n} \mathbf{1}_{\{\rho_{u-}^n < 1\}}\mathbf{1}_{\{u<t\}} \mathbf{1}_{\{x > t-u\}},
$$
is a predictable process. 
\begin{remark}
Note that $\rho^n$ has paths of finite variation on compacts. Indeed, the process $\rho^n$ is piecewise constant with jumps corresponding to arrivals according to $A^n$ only if the current state $\rho^n$ is less than one. This means that the total variation of $\rho^n$ is bounded by that of $A^n$ which is a non-homogeneous Poisson process and hence is of finite variation. In addition, $\rho^n$ is adapted and c\`adl\`ag, and consequently by \cite[Theorem 26]{protter} $\rho^n$ is a pure jump quadratic semimartingale.
 \end{remark}
We now state a functional fluid limit for $\rho^n$ as $n$ tends to infinity.

\begin{theorem}\label{thm:fluid-lim}
Let the conditions in Assumption~\ref{assum:model} hold. Then for any $T>0$ we have:
\beq\label{eq:fluid-lim}
\lim_{n \to \infty} \sup_{t \in [0,T]} \lln \rho_t^n - \rho_t \rrn = 0, \textnormal{almost surely},
\eeq      
where $\rho$ is the solution to a non-linear Volterra integral equation:
\beq\label{eq:rho}
\rho_t = \int_0^t \mathbf{1}_{\{\rho_{u-}< 1\}} \bar{F}(t-u) \la_u du \quad\text{   for   } t> 0\quad\text{ and }\rho_0=0.
\eeq
\end{theorem}

\begin{proof}
Recall that the counting measure $p^R$ posseses a compensator $p_c^R$ given by \eqref{eq:compensator}. Now, observe that from \eqref{eq:rho-1} and \eqref{eq:p*} the fraction of occupied servers $\rho^n$ has the following decomposition:
\begin{align*}
\rho_t^n &= \int_0^t \int_{\R} W_n({t},u,x)p_{\ast}^R(du, dx) + \int_0^t \int_{\R} W_n({t},u,x) p_c^R(du, dx)\\
	      &= \int_0^t \int_{\R} W_n({t},u,x)p_{\ast}^R(du, dx) + \int_0^t \mathbf{1}_{\{\rho_u^n < 1\}} \bar{F}(t-u) \la_u du\\
	      &=: X_t^n + \ga_t^n,
\end{align*}
where $p_{\ast}^R$ is given by \eqref{eq:p*} and $X^n$ is another bounded cadlag semimartingale. Indeed, $X^n$ is the difference of a bounded cadlag semimartingale $\rho^n$ and the bounded finite variation process $\ga^n$. In fact, we have
$$
\sup_n \sup_t |X_t^n| \leq 1+ \int_0^T \la_u du,
$$
where the second quantity is finite because of our assumption that $\la$ is locally integrable according to Assumption~\ref{assum:model}. In addition we have almost surely
$$
\lim_{\der \downarrow 0} \sup_n w_{X^n}^{'} (\der) = 0,
$$
where $w_{x}^{'}(\der) = \inf_{t_i} w_x[t_i, t_{i+1})$. This follows from the facts that $w_{\rho^n}^{'}(\der) = 0$ (almost surely $\rho^n$ has finitely many jumps in $[0,T]$ and is constant in between), and the fact that for all $n$ one must have $\lim_{\der \downarrow 0} \sup_n w_{\ga^n}(\der) = 0$. In order to obtain this last assertion observe that $|\ga_t^n - \ga_s^n| \leq 2 \int_{s}^t \la_u du \leq 2 w_{\laa}(|s-t|)$ where $\laa(t) = \int_0^t \la_u du$ is continuous on $[0,T]$ and hence also uniformly continuous.

We thus have that $\{X^n\}_{n \geq 1}$ has compact closure in the Skorohod topology. In other words we have obtained tightness.

Next, fix $t$ and obtain (cf. \cite[Theorem 2.3.7]{applebaum}):
\begin{multline*}
\E {\lp X_t^n \rp}^2 = \E \lc \int_0^t \int_{\R} W_n^2(t,u,x) p_{\ast}^R(du, dx) \rc  \leq \dfrac{1}{n^2} \int_0^t \int_{\R} p_c^R(du,dx) = \dfrac{1}{n} \int_0^t \la_u du
\leq \dfrac{1}{n} \laa \longrightarrow 0. 
\end{multline*}
Consequently for each $t \in [0,T]$, $X_t^n \stackrel{p}{\rightarrow} 0$. Thus for any $(t_1,t_2,\ldots,t_d) \in [0,T]^d$, the finite dimensional vectors $(X_{t_1}^n, \ldots, X_{t_n}^n) \stackrel{p}{\longrightarrow} (0, \ldots, 0)$ as a consequence of the Cramer-Wold device~\cite[Theorem 7.7]{billingsley}. Recalling the tightness condition we have thus obtained that $X^n$ converges in distribution to the constant zero function and hence also in probability. Since the limiting function is non-random, the convergence is also in probability under the uniform topology. Thus we have:
$$
X^n \stackrel{ucp}{\longrightarrow} 0.
$$

\noindent
Observe that we have
\beq\label{eq:rho^n}
\rho_t^n = X_t^n + \ga_t^n = X_t^n + \int_0^t \mathbf{1}_{\{\rho_{u-}^n < 1\}} \bar{F}(t-u) \la_u du.
\eeq
For every $\om \in \oom$, $\rho^n$ by definition belongs to the Skorohod space $D[0,T]$. {Furthermore there exists $\oom_1 \subset \oom $} such that $P(\oom_1)=1$ and for every $\om \in \oom_1$ the sequence $\{\rho^n(\om)\}_{n \geq 1}$ has compact closure in the Skorohod topology because 
\begin{enumerate}
\item[(i)] $\sup_{n} \sup_{t} \lln \rho_t^n \rrn \leq 1$.
\item[(ii)] $\lim_{\der \downarrow 0} \sup_{n} w_{\rho^n}^{'}(\der) = 0$,
\end{enumerate}
where 
$$
w_{x}^{'} (\der) = \inf_{t_i} w_{x} [t_i, t_{i+1}),
$$
and $w_x$ is the modulus of continuity of $x$ in $[t_i, t_{i+1})$. Note that item (ii) is true almost surely because almost surely $\rho$ will have finitely many jumps in the time horizon $[0,T]$.

Consider any $\om \in \oom_1$. The above considerations thus show that for every subsequence $n_k$ of the naturals, $\rho^{n_k}$ has a convergent subsequence which converges to an element of $D$. That is, there exists a subsequence $\{m_k\} \subset \{n_k\}$ such that 
\beq\label{eq:rho^n-conv}
\rho^{m_k} \rightarrow \rho,~ \text{in the Skorohod topology}.
\eeq 
 Henceforth, we try to identify $\rho$. To that attempt, we give a slightly different representation of $\rho^n$. Observe that the set $\{\rho^n_{u-} <  1\}$ is identical to the set $\{ \rho^n_{u-} \leq 1-\frac{1}{n}\}$. This is because $\rho^n$ only takes values in $\{\frac{i}{n}: i=1,\ldots,n \}$. This gives us the opportunity to replace the indicator $\mathbf{1}_{\{\rho_{u-}^n < 1\}}$ in \eqref{eq:rho^n} by a smooth approximation. In particular consider a sequence of smooth functions $\mathbf{1}^{d}:\R \to [0,1]$ for $d \in (0,1)$ such that 
$$
\mathbf{1}^d(x) = 
\begin{cases}
1, &\text{for }x\leq 1- \frac{2d}{3},\\
0 &\text{for }x \geq 1-\frac{d}{3}.
\end{cases}
$$
In addition let $\underline{\mathbf{1}}^{d} : \R \mapsto [0,1]$ for $d \in (0,1)$ de defined as:
$$
\underline{\mathbf{1}}^{d} (x) = 
\begin{cases}
1 &\text{for }x \leq 1 -d\\
0 &\text{for }x > 1-d.
\end{cases}
$$
Using this notation we can replace $\mathbf{1}_{\{\rho_{u-}^n < 1\}}$ in \eqref{eq:rho^n} by $\mathbf{1}^{\frac{1}{n}}(\rho_{u-}^n)$ as both the quantities are the same. Thus our alternate representation of $\rho^n$ is given by:
$$
\rho_{t}^n = X_t^n + \int_0^t \mathbf{1}^{\frac{1}{n}}(\rho_{u-}^n) \bar{F}(t-u) \la_u du.
$$
Fix any arbitrary $t \in [0, T]$. Since $L^1[0,t]$ is a separable Banach space with dual $L^{\infty}[0,t]$ and $\mathbf{1}^{\frac{1}{m_k}}(\rho^{m_k})$ is bounded in $L^{\infty}[0,t]$, there is a subsequence $\{l_k\} \subset \{m_k\}$, where $m_k$ is as in \eqref{eq:rho^n-conv} such that
\beq\label{eq:weak*}
\lim_{k \to \infty} \int_0^t h(u) \mathbf{1}^{\frac{1}{l_k}}(\rho_{u-}^{l_k}) du = \int_0^t h(u) w(u) du.
\eeq
for every $h \in L^1[0,t]$. In particular, let us consider the function $h$ given by:
$$
h(u) = \bar{F}(t-u)\la_u.
$$
Observe that our assumption on $\la$ ensures that this specific $h$ lies in $L^1[0,t]$. We thus have:
$$
\lim_{k \to \infty} \int_0^t  \mathbf{1}^{\frac{1}{l_k}}(\rho_{u-}^{l_k})\bar{F}(t-u)\la_u = \int_0^t w(u)\bar{F}(t-u)\la_u du,
$$
Recall that $X^n \stackrel{ucp}{\longrightarrow} 0$. Consequently there exists a subsequence $\{r_k\} \subset \{l_k\}$, (which we conveniently choose to be a subset of $\{l_k\}$) such that 
$$
{\|X^{r_k}\|}_{T} \longrightarrow 0,
$$
for every $\om \in \oom_2$, where $\oom_2 \subset \oom_1$ satisfies $P(\oom_2) = 1$. For this sequence $\{r_k\}$ we thus obtain:
$$
\lim_{k \to \infty} \rho^{r_k}_t = \lim_{k \to \infty} \lp X^{r_k}_t + \int_0^t \mathbf{1}^{\frac{1}{r_k}}(\rho_{u-}^{r_k}) \bar{F}(t-u)\la_u \rp = \int_0^t w(u)\bar{F}(t-u)\la_u du.
$$
Due to \eqref{eq:rho^n-conv} we must then have:
$$
\rho_t = \int_0^t w(u) \bar{F}(t-u)\la_u du.
$$
Observe that the above representation guarantees that $\rho$ is continuous and the convergence stated in \eqref{eq:rho^n-conv} is in the uniform topology for each $\om \in \oom_2$. Consequently fix $\ep > 0$ and choose $N$ large enough such that for all $k > N$ we have $r_k > \frac{3}{\ep}$ and ${\|\rho^{r_k} - \rho\|}_T < \frac{\ep}{3}$. Then it is readily checked that
$$
\underline{\mathbf{1}}^{\ep} (\rho_{u-}) \leq \mathbf{1}^{\frac{1}{r_k}}(\rho_{u-}^{r_k}) \leq \mathbf{1}_{\{\rho_{u-} < 1\}}.
$$
Consider any $h \geq 0$ such that $h \in L^1[0,T]$. Let us multiply each side of the above equation by $h$ and integrate. We thus obtain:
$$
\int_0^t h(u) \underline{\mathbf{1}}^{\ep} (\rho_{u-})du \leq \int_0^t h(u) \mathbf{1}^{\frac{1}{r_k}}(\rho_{u-}^{r_k})du \leq \int_0^t h(u) \mathbf{1}_{\{\rho_{u-} < 1\}}du.
$$
Note that 
$$
\lim_{d \downarrow 0} \underline{\mathbf{1}}^{d}(x) = \lim_{d \downarrow 0} \mathbf{1}^d(x) = \mathbf{1}_{\{x < 1\}} .
$$
Consequently taking $k \to \infty$ (this is okay as we need $k > N=N(\ep)$) and then $\ep \downarrow 0$ we have by the dominated convergence theorem and  \eqref{eq:weak*} that:
$$
\int_0^t h(u) \mathbf{1}_{\{ \rho_{u-} < 1 \}} du \leq \int_0^t w(u) h(u) du \leq \int_0^t h(u) \mathbf{1}_{\{ \rho_{u-} < 1 \}} du.
$$
We now take $h= \bar{F}(t-\cdot)\la$ and thus we conclude that
$$
\rho_t = \int_0^t \mathbf{1}_{\{\rho_{u-} < 1\}}\bar{F}(t-u) \la_u du.
$$
Observe that our considerations above hold for any $t \in [0,T]$ and hence \eqref{eq:rho} holds true for any $t \in [0,T]$.
\end{proof}

\begin{remark}
Our considerations so far as stated in Assumption~\ref{assum:model} are constrained on systems which start empty, that is, $\rho_0^n = 0$, for all $n$. However this is easily relaxed as stated in the following corollary which holds under the following assumption
\end{remark}

\begin{assumption}\label{assum:model-2}
Let the conditions under Assumption~\ref{assum:model} hold. In addition let the initial fraction of occupied servers $\rho_0^n$ satisfy the following asymptotic result:
$$
\lim_{n \to \infty} \lln \rho_0^n - r_0 \rrn = 0, \quad \textnormal{almost surely,}
$$
where $r_0 \in (0,1]$. Moreover, assume that the remaining service times for each of these occupied servers are iid drawn from a distribution $G$.
\end{assumption}

\begin{corollary}\label{cor:fluid-lim-2}
Let the conditions in Assumption~\ref{assum:model-2} hold. Then for any $T>0$ we have:
\beq\label{eq:fluid-lim-2}
\lim_{n \to \infty} \sup_{t \in [0,T]} \lln \rho_t^n - \rho_t \rrn = 0, \quad \textnormal{almost surely},
\eeq      
where $\rho$ is the solution to a non-linear Volterra integral equation:
\beq\label{eq:rho-2}
\rho_t = \rho_0 \bar{G}(t) + \int_0^t \mathbf{1}_{\{\rho_{u-}< 1\}} \bar{F}(t-u) \la_u du \quad\text{   for   } t> 0\quad\text{ and }\rho_0=r_0.
\eeq
\end{corollary}
\begin{proof}
Let the initial number of occupied servers be $N_0^n$, so that $\rho_0^n = \frac{N_0^n}{n}$. Let the remaining service times for these $N_0^n$ many jobs be $(S_i^0)_{1\leq i \leq N_0^n}$. Then $\rho_t^n$ can be represented as:
$$
\rho_t^n = \dfrac{1}{n}{\sum_{i=0}^{N_0^n} \mathbf{1}_{\{ S_i^0 > t \}}} + \int_0^t \int_{\R} W_n(t,u,x) p^R(du, dx).
$$
Observing that $N_0^n$ goes to infinity since $r_0 > 0$, we can represent this as follows:
$$
\rho_t^n = {\sum_{i=0}^{N_0^n} \mathbf{1}_{\{ S_i^0 > t \}}}{N_0^n} \rho_0^n + \int_0^t \int_{\R} W_n(t,u,x) p^R(du, dx).
$$
We have already analyzed the second quantity on the right hand side in Theorem~\ref{thm:fluid-lim}. Now using Assumption~\ref{assum:model-2} we have that the first quantity converges almost surely. In particular, we have:
$$
\lim_{n \to \infty} \lln \dfrac{\sum_{i=0}^{N_0^n} \mathbf{1}_{\{ S_i^0 > t \}}}{N_0^n} \rho_0^n - \rho_0 \bar{G}(t) \rrn = 0, \quad \text{almost surely}.
$$
This completes the proof.
\end{proof}

\begin{figure}[h!]
\centering
\includegraphics[width=\linewidth]{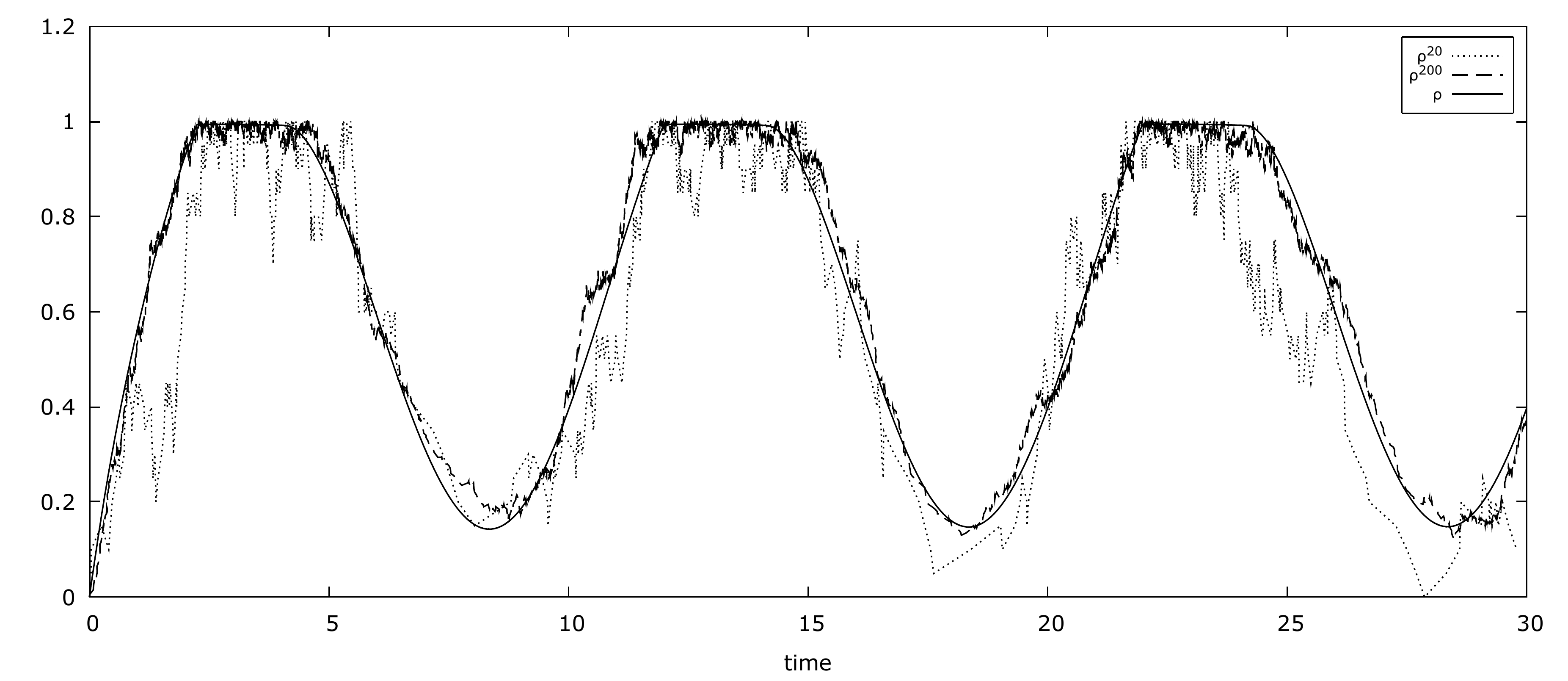}
\caption{\small Convergence of $\rho^n$ to $\rho$ for an initially empty system.}
\par
\scriptsize $\rho^n$ has been simulated for $n=20$ and $n=200$ with service times drawn from $\textrm{Lognormal}(-0.5, 1)$ while the intensity of arrivals is sinusoidal with $\la_u = \frac{2}{3}(1+\sin(\frac{2\pi u}{10}))$. $\rho$ has been approximated using a mollified version of the indicator function in relation \eqref{eq:rho} (note $\rho$ doesn't reach $1$ as a consequence).
\end{figure}

\subsection{Existence and uniqueness of fluid limit.}
In this subsection we prove the existence and uniqueness of the fluid limits $\rho$ and $\Theta$. 
\begin{theorem}\label{thm:exi+uniq}
For all $r_0 \in [0,1]$ there exists a unique solution to the non-linear Volterra equation given by \eqref{eq:rho-2}. 
\end{theorem}
\begin{proof}
Existence of a solution is well known and its proof is very similar to what we have presented in the proof to Theorem~\ref{thm:fluid-lim}. Namely, we mollify the discontinuous coefficient $\mathbf{1}_{\{\cdot < 1\}}$ by a smooth version, use existence results for smooth coefficients and then show that the limit satisfies \eqref{eq:rho-2}. See \cite{kiffe} for the existence result in a more general setup, and for a more general definition of solution to nonlinear Volterra equations with discontinuous coefficient. Now
we will show that for all $T>0$, \eqref{eq:rho-2} has a unique solution for $t \in [0,T]$. 
We first show that $\rho$ given by:
$$
\rho_t = \rho_0 \bar{G}(t) + \int_0^t \mathbf{1}_{\{ \rho_{u-} < 1 \}} \bar{F}(t-u) \la_u du,
$$
takes values in $[0,1]$. The fact that $\rho$ is positive is immediate by the positivity of $\rho_0$, $\bar{F}$, $\bar{G}$ and $\la$. Suppose there exists a $t_0 \in [0,T]$ such that $\rho_{t_0} > 1$. Since $\bar{G}$ is non-increasing and 
$$\int_0^t \mathbf{1}_{\{\rho_{u-} < 1\}}\bar{F}(t-u)\la_u du$$ is continuous as a function of $t$, the jumps of $\rho$ if any are negative. Thus there must exist an $s_0 < t_0$ such that $\rho_{s_0} = 1$ and $\rho_s \geq 1$ for $s \in (s_0, t_0]$. Consequently we must have:
$$
\rho_{t_0} = \rho_0 \bar{G}(t_0) + \int_0^{t_0} \mathbf{1}_{\{\rho_{u-}<1\}} \bar{F}(t_0-u) \la_u du =  \rho_0 \bar{G}(t_0) + \int_0^{s_0} \mathbf{1}_{\{\rho_{u-}<1\}} \bar{F}(t_0-u) \la_u du.
$$
However, since $\bar{G}$ and $\bar{F}$ are non-increasing, we have:
$$
 \rho_0 \bar{G}(t_0) + \int_0^{s_0} \mathbf{1}_{\{\rho_{u-}<1\}} \bar{F}(t_0-u) \la_u du \leq  \rho_0 \bar{G}(s_0) + \int_0^{s_0} \mathbf{1}_{\{\rho_{u-}<1\}} \bar{F}(s_0-u) \la_u du = \rho_{s_0}=1
$$
Thus we obtain the contradiction that $\rho_{t_0} \leq 1$. Having obtained that $\rho$ takes values in $[0,1]$ it is easy to  obtain the following hitting times to $1$ and exit times from $1$ for a solution $\rho$. We denote:
\begin{align*}
\text{if }\rho_0 < 1 &\text{ then }\si_0 = 0, \text{ else }\si_0 = \inf \lcl t>0 : \rho_0 \bar{G}(t) < 1 \rcl\\
\tau_1 &= \inf \lcl t > 0: \rho_0 \bar{G}(t) + \int_0^t \bar{F}(t-u)\la_u du = 1 \rcl,\\
\si_1 &= \inf \lcl t > \tau_1: \rho_0 \bar{G}(t) + \int_0^{\tau_1}\bar{F}(t-u) \la_u du < 1 \rcl.
\end{align*}
Here $\tau_1$ denotes the first hitting time from below for a solution $\rho$ and $\si_1$ ($\si_0$) denotes the first exit time from $1$ when the initial condition $\rho_0 < 1$ ($\rho_0 = 1$). The next set of hitting and exit times are defined similarly. For $k \geq 2$ denote:
\begin{multline*}
\tau_k = \inf \lcl t > \si_{k-1}: \rho_0 \bar{G}(t) + \int_{I_{k,t}} \bar{F}(t-u) \la_u du = 1\rcl, \text{ where } I_{k,t} = \cup_{i=1}^{k-1} [\si_{i-1},\tau_i)\cup[\si_{k-1},t).
\end{multline*}
and 
\beq\label{eq:si_k}
\si_k = \inf \lcl t > \tau_{k}: \rho_0 \bar{G}(t) + \int_{{J}_{k}} \bar{F}(t-u) \la_u du < 1 \rcl, \text{ where } J_{k} = \cup_{i=1}^{k} [\si_{i-1},\tau_i).
\eeq
In addition the specific solution $\rho$ can now be actually represented as 
\beq\label{eq:rho-J_kt}
\rho_t = \rho_0 \bar{G}(t) + \int_{J_{t}} \bar{F}(t-u) \la_u du,
\eeq
where 
$$
J_{t} = \cup_{i=1}^{\infty} [\si_{i-1},\tau_i) \cap[0,t].
$$
Let us justify the above representation rigorously. Consider first the case when $G$ is continuous. This yields that $\rho$ must also be continuous. Now observe that the interval $[0,1)$ is open in $[0,1]$ and as such the pre-image of $[0,1)$ with respect to $\rho$:
$$
\rho^{-1} [0,1) = \lcl t \in [0, \infty): \rho_t \in [0,1) \rcl,
$$
is an open set of $[0, \infty)$. Consequently this pre-image can be represented as countable union of open intervals in $[0,\infty)$:
\beq\label{eq:pre-image-decomp}
\rho^{-1} [0,1) = \cup_{k=1}^{\infty} L_k,
\eeq
where $L_k$'s are open intervals in $[0, \infty)$.
If $G$ however is not continuous, the fact that it is right continuous and non-decreasing guarantees, as we have already mentioned before, that $\bar{G}$ has at most countably many jumps of negative size. The addition of this complexity doesn't complicate the pre-image too much. We just have at most countably many $L_k$'s in \eqref{eq:pre-image-decomp} replaced by left-closed right-open intervals, that is:
$$
\rho^{-1} [0,1) = \cup_{k=1}^{\infty} \tilde{L}_k,
$$
where each $\tilde{L}_k$ is either an open or a left-closed right-open interval of $[0, \infty)$. In our considerations above we have denoted $J_{t}$ to be:
$$
J_{t} = \lp \cup_{i=1}^{\infty} \tilde{L}_k \rp \cap [0,t].
$$
Note that for the purposes of obtaining the solution from $J_{t}$ using equation~\eqref{eq:rho-J_kt} we may replace the open intervals in $\{\tilde{L}_k\}_{k \geq 1}$ by left-closed right-open intervals without affecting the solution because the solution would be continuous at the left limit point of the said interval. We have thus obtained a one-one correspondence between solutions of \eqref{eq:rho-2} and the corresponding intervals $\tilde{L}_k = [\si_{k-1}, \tau_k)$ through relation \eqref{eq:rho-J_kt}. Now suppose \eqref{eq:rho-2} admits two solutions $\rho^1$ and $\rho^2$. By the one-one correspondence established these two solutions will differ only if they admit two different countable collection of intervals $\{\tilde{L}_k^1\}_{k \geq 1}$ and $\{\tilde{L}_k^2\}_{k \geq 1}$. Let 
$$
l = \min\{k : \tilde{L}_k^1 \neq \tilde{L}_k^2\}
$$
If $l=1$, deriving a contradiction is straightforward. Indeed, if $\rho_0 < 1$, then it is immediate that the solution would be unique until the first time it hits $1$, and consequently $\si_0=0$ and $\tau_1$ are unique. Similarly, if $\rho_0 = 1$, $\si_0 = \inf\{ t > 0: \rho_0 \bar{G}(t) < 1 \}$, which is unique and by translation one can obtain a new equation on $[\si_0, \infty)$ as follows:
$$
\ga_s = \rho_0 \bar{G}(s+ \si_0) + \int_{\si_0}^s \bar{F}(s-u) \la_u du,
$$
which also admits a unique solution until it hits $1$. Thus in both cases, the first interval $\tilde{L}_1$ is determined by $F$, $G$ and $\la$, and hence unique. The argument for the latter case can be modified and applied to derive a contradiction when $l > 1$. In this case $\si_{l-1}$ is given by \eqref{eq:si_k} with $J_{l-1} = \cup_{i=1}^{l-1} \tilde{L}_i$, and is hence unique. We therefore translate our equation to $[\si_{l-1}, \infty)$ to obtain like before:
$$
\ga_s = \rho_0 \bar{G}(s+ \si_{l-1}) + H(s) + \int_{\si_{l-1}}^s \bar{F}(s-u) \la_u du,
$$
where $H(s) = \int_{J_{l-1}} \bar{F}(s-u) \la_u du$. Again, this admits a unique solution until the first time it hits $1$, and hence $\tilde{L}_1$ is also unique. This provides our required contradiction.
\end{proof}

\subsection{Related processes}
\subsubsection{Integrated fraction of occupied servers}
Let $\tte^n$ denote the integrated fraction of occupied servers, that is, for $t > 0$:
$$
\tte_t^n = \int_0^t \rho_u^n du.
$$
Observe that $t - \tte^n_t = \int_0^t (1-\rho_u^n) du$ is the {cumulative idleness (in the sense that this counts the time instants when any server is idle)} since $\rho^n_u = 1$ only if all the servers are occupied. Indeed, the fluid limit in~\eqref{eq:tte} below provides a relation between the mean service times, arrival intensity and the integrated process. 

It is readily seen that $\tte^n$ has an integral representation with respect to the counting measure $p^R$:
\beq\label{eq:tte^n-decomp}
\tte_t^n = \dfrac{1}{n} \sum_{i=1}^{\infty} \mathbf{1}_{\{ \rho_{T_i}^n < 1 \}} \mathbf{1}_{\{T_i < t\}} S_i \wedge (t-T_i)
= \int_0^t \int_{\R} V_n(t,u,x) p^R(du, dx),
\eeq
where $V_n(t,u,x) = \frac{1}{n} \mathbf{1}_{\{\rho_{u-}^n < 1\}} \mathbf{1}_{\{u < t\}} (x \wedge (t-u))$. Similar to Theorem~\ref{thm:fluid-lim} we now have the following fluid limit for the integrated process $\Theta^n$.
\begin{theorem}\label{thm:int-proc}
Let the conditions in Assumption~\ref{assum:model-2} hold. Then for any $T>0$ we have:
$$
\lim_{n \to \infty} \sup_{t \in [0,T]} \lln \Theta_t^n - \Theta_t \rrn = 0, \quad \textnormal{almost surely},
$$
where $\tte$ for $t>0$ is given by
$$
\tte_t = \int_0^t \rho_u du,
$$
and where $\rho$ is given by \eqref{eq:rho-2}.
\end{theorem}

\begin{remark}
Observe that \eqref{eq:rho-2} implies that $\tte$ has the alternate more explicit expression:
\beq\label{eq:tte}
\tte_t = \rho_0 \int_0^t E[S^0 \wedge t] du + \int_0^t \mathbf{1}_{\{\rho_{u-}<1\}} E [S\wedge (t-u)]\la_u du,
\eeq
where $S^0$ is distributed as $G$, while $S$ is distributed as $F$.
This is readily obtained by integrating the right hand side of \eqref{eq:rho-2}.
\end{remark}

\begin{proof} By Corollary~\ref{cor:fluid-lim-2} we have that
$$
\lim_{n \to \infty} \sup_{t \in [0,T]} \lln \rho_t^n(\om) - \rho_t \rrn = 0,
$$
for all $\om \in \oom_1$ such that $\oom_1 \subset \oom$ and $P(\oom) = 1$. Consequently fix $\om \in \oom_1$ and $\ep > 0$ to obtain $N(\om)$ large enough such that
$$
\lln \rho_t^n(\om) - \rho_t \rrn < \dfrac{\ep}{T},
$$
for all $n > N(\om)$. Thus we have 
$$
\lln \int_0^t \rho_u^n du - \int_0^t \rho_u du \rrn \leq \int_0^t \lln \rho_u^n - \rho_u \rrn du \leq \ep,
$$
for all $n > N(\om)$. This completes the proof.

\noindent
\emph{Alternate explicit proof.}
From \eqref{eq:tte^n-decomp} we have the following decomposition 
\begin{align*}
\tte_t^n &= \int_0^t \int_{\R} V_n(t,u,x) p_{\ast}^R (du, dx) + \int_0^t \int_{\R} V_n(t,u,x) p_{c}^R (du, dx)\\
&= \int_0^t \int_{\R} V_n(t,u,x) p_{\ast}^R (du, dx) + \int_0^t \int_{\R} \mathbf{1}_{\{ \rho_{u-}^n < 1 \}} \mathbf{1}_{\{ u<t \}} (x \wedge (t-u)) \nu(x) \la_u dx du\\
&= \int_0^t \int_{\R} V_n(t,u,x) p_{\ast}^R (du, dx) + \int_0^t \mathbf{1}_{\{\rho_{u-}^n < 1\}} \be[X \wedge (t-u)] \la_u du\\
&= X_t^{1,n} + \ga_t^{1,n},
\end{align*}
where $X^{1,n}$ is a semimartingale bounded above by $\Theta^n$ (since $\ga^{1,n}$ is positive) and in turn by $T$ (since we are in a finite time horizon $T$ and $|\rho^n| \leq 1$).

Also, notice that $\Theta^n$ is a continuous function and since $\rho^n \leq 1$, its modulus of continuity satisfies:
$$
w_{\Theta^n}(\der) := \sup_{|s-t|<\der} |\int_s^t \rho_u^n du| \leq \der
$$
In addition, the modulus of continuity of $\ga^{1,n}$ satisfy:
$$
w_{\ga^{1,n}}(\der) \leq \sup_{|s-t| \leq \der} 2\int_s^t \E[X] \la_u du \leq 2 w_{\laa^{\ast}}(\der),
$$
where $\laa^{\ast}$ is given by
$$
\laa^{\ast} (t) = \int_0^t \E[X] \la_u du,
$$
which is uniformly continuous on $[0,T]$. As a consequence of all this we have:
$$
\lim_{\der \downarrow 0} \sup_n w_{X^{1,n}}(\der) = 0.
$$
Observe that the $L^2$-norm of the semimartingale $X^{1,n}$ satisfies:
\begin{align*}
\E {\lp X_t^{1,n} \rp}^2 &= \E \lc \int_0^t V_n^2(u,x) p_c^R(du,dx) \rc\\
&\leq \dfrac{1}{n^2} \int_0^t \int_{\R} (x \wedge (t-u))^2 p_c^R(du, dx)\\
&\leq \dfrac{1}{n} \E[X^2] \int_0^t \la_u du \leq \dfrac{1}{n} \E[X^2] \laa  {\longrightarrow} 0.
\end{align*}
Consequently for each $t \in [0,T]$, $X_t^{1,n} \stackrel{p}{\rightarrow} 0$. Thus the finite dimensional vectors $(X_{t_1}^{1,n}, \ldots, X_{t_n}^{1,n}) \stackrel{p}{\longrightarrow} (0, \ldots, 0)$. Recalling the tightness condition we have thus obtained that $X^{1,n}$ converges in distribution to the constant zero function and hence also in probability. Since the limiting function is continuous, the convergence is also in the uniform topology. Thus we have:
$$
X^{1,n} \stackrel{ucp}{\longrightarrow} 0.
$$
By our previous considerations since $\rho^n \stackrel{ucp}{\longrightarrow} \rho$ we have 
$$
\ga^{1,n} \stackrel{ucp}{\longrightarrow} \ga^1,
$$ where $\ga^1$ is given by  $\ga_t^{1} = \int_0^t \mathbf{1}_{\{\rho_{u-}<1\}} \be[X\wedge(t-u)] \la_u du$.
We have thus obtained 
$$
\Theta^{n} \stackrel{ucp}{\rightarrow} \ga^1.
$$
A similar trick as employed in the previous section guarantees almost sure convergence. This is because for any subsequence there is a further subsequence such that the convergence of $X^{1,n}$ and $\ga^{1,n}$ happen almost surely. In addition, there exists an $\oom_3 \subset \oom$ with $P(\oom_3) = 1$ such that for every $\om \in \oom_3$, the sequence $\{\Theta^{n}(\om)\}_{n \geq 1}$ is tight. Similar calculations as employed in the previous section now yield:
$$
\Theta^{n} {\longrightarrow} \ga^1~\textnormal{almost surely,}
$$
with uniform convergence over $[0,T]$.
\end{proof}

The following regarding the integrated process is an easy corollary of Theorem~\ref{thm:exi+uniq}.
\begin{corollary}
There exists an unique solution $\tte$ to \eqref{eq:tte}, where $\rho$ is given by \eqref{eq:rho-2}.
\end{corollary}

\subsubsection{Blocked arrivals} Blocking probabilities and congestion measurement are the most important measures of performance in loss models. Computing these quantities in non-stationary models is rather hard to do requiring approximations~\cite{pender1,whitt2}. As a direct consequence of Theorem~\ref{thm:fluid-lim}, our next result establishes a fluid limit for the {fluid-scaled} cumulative number of blocked arrivals, $$b^n_t := \frac{1}{n} \sum_{i=1}^{\infty} \mathbf 1_{\{\rho^n_{T_i} = 1\}} \mathbf 1_{\{T_i \leq t\}}.$$ Note that $n b^n_t$ counts the number of arrivals by $t > 0$ that encountered a fully occupied system on arrival. 
\begin{theorem}
	Let the conditions of Theorem~\ref{thm:fluid-lim}. Then for any $T > 0$ we have:
		$$
			\lim_{n\to\infty} \sup_{t \in [0,T} |b^n_t - b_t| = 0 ~\quad \textnormal{almost surely,}
		$$
		where $b$ for $t > 0$ is given by
		$$
			b_t = \int_0^t \mathbf 1_{\{\rho_u = 1\}} \lambda_u du
		$$
		and $\rho$ is given by~\eqref{eq:rho-2}.
\end{theorem}

We omit the proof as it closely follows that of Theorem~\ref{thm:fluid-lim}.  Observe that the ratio
$$
\frac{b_t}{\Lambda_t} = \frac{\int_0^t \mathbf 1_{\{\rho_u = 1\}} \lambda_u du}{\int_0^t \lambda_u du}
$$
can be used as a measure of system congestion or an approximation of the likelihood of being blocked on arrival at the queue.

\section{Conclusions}
The results in this note complement extant results establishing many-server fluid limits, by considering the nonstationary, non-Markovian loss model setting, and by using a bespoke semimartingale representation of the fraction of occupied servers. The primary result shows that the fraction of  occupied servers converges to the unique solution of a Volterra integral equation. As a consequence of our main result, we also establish a fluid limit to the integrated number of occupied servers and the fraction of arriving jobs that are blocked on arrival. We anticipate our results, proofs of which are quite simple, should be broadly useful. In future work we anticipate extending the analysis to include functional central limit theorems. However, the analysis is much harder than the fluid limit results in this note and merit a separate paper.

\section{Acknowledgements}
PC gratefully acknowledges support from a Purdue Research Foundation dissertation fellowship. HH was partly supported by the National Science Foundation through grant CMMI/1636069. 
\section{References}
\bibliographystyle{elsarticle-num-names}
\bibliography{references}

\end{document}